\documentclass[11pt]{amsart}
\usepackage{amssymb,amsmath}
\usepackage{color}
\usepackage[nobysame]{amsrefs}

\def\Bbb{\mathbb}
\def\Cal{\mathcal}
\def\Dt{\partial_t}

\def\eb{\varepsilon}

\def\R {\mathbb{R}}

\def\<{\left<}
\def\>{\right>}

\def\Nx{\px}
\def\Dx{\px^2}
\def\({\left(}
\def\){\right)}
\def\Cal{\mathcal}
\def\Bbb{\mathbb}

\newtheorem{proposition}{Proposition}[section]
\newtheorem{theorem}[proposition]{Theorem}

\newtheorem{lemma}[proposition]{Lemma}

\theoremstyle{definition}
\newtheorem{definition}[proposition]{Definition}

\newtheorem{remark}[proposition]{Remark}

\numberwithin{equation}{section}

\def \no#1#2#3 {{\bf #1} (#3), #2.}
\def \eds#1#2#3 {#1, #2, #3.}

\title[Inertial manifolds for convective RDEs] {Inertial manifolds for 1D  reaction-diffusion-advection systems. Part I: Dirichlet and Neumann boundary conditions}
\author[A. Kostianko and S. Zelik]{ Anna Kostianko${}^1$ and Sergey Zelik${}^1$}
\address{${}^1$
University of Surrey, Department of Mathematics,
Guildford, GU2 7XH, United Kingdom, a.kostianko@surrey.ac.uk, s.zelik@surrey.ac.uk.}
\thanks{This work is partially supported by  the grants  14-41-00044 and 14-21-00025 of RSF as well as  grants 14-01-00346  and 15-01-03587 of RFBR}

\begin{document}
\begin{abstract} This is the first part of our study of  inertial manifolds for the system of 1D reaction-diffusion-advection equations which is devoted to the case of  Dirichlet or Neumann  boundary conditions. Although this problem does not initially possess the spectral gap property, it is shown that this property is satisfied after the proper non-local change of the dependent variable. The case of periodic boundary conditions where the situation is principally different and the inertial manifold may not exist is considered in the second part of our study.
\end{abstract}
\maketitle
\tableofcontents
\def\R{\Bbb R}
\def\Dt{\partial_t}
\def\Dx{\partial^2_x}
\def\Nx{\partial_x}
\def\eb{\varepsilon}

\section{Introduction}\label{s0}
It is believed that the long-time behavior of many dissipative PDEs in bounded domains is essentially finite-dimensional. Thus, despite of the infinite-dimensionality of the initial phase space, the reduced dynamics on the so-called global attractor can be effectively described by finitely many parameters. This conjecture is partially supported by the fact that this global attractor usually has finite Hausdorff and box-counting dimensions and, by the Man\'e projection theorem, can be embedded by a H\"older continuous homeomorphism into the finite-dimensional plane of the phase space. In turn, this allows us to describe the reduced dynamics on the attractor in terms of a finite system of ODEs - the so-called inertial form of the PDE considered, see \cites{bv1,cv,he,tem,MirZe,28,Zel2} and references therein.
\par
Unfortunately, the reduction based on the Man\'e projection theorem guarantees only the H\"older continuity of the vector field in the above mentioned inertial form although the regularity of this vector field seems to be crucial here. Indeed, as recent counterexamples show (see \cites{EKZ,Zel2}), this vector field cannot be made Lipschitz or log-Lipschitz continuous in general and this lack of regularity may lead to actual infinite-dimensionality of the reduced dynamics on the attractor despite the fact that the attractor has finite box-counting dimension. By this reason, understanding under what extra assumptions the considered PDE possesses an inertial form with more regular vector field becomes a central problem of the theory.
\par
An ideal situation arises when the considered PDE possesses the so-called {\it inertial manifold} (IM) which is finite-dimensional invariant manifold of the phase space with exponential tracking (asymptotic phase) property which contains the global attractor. In this case, the desired inertial form is constructed by restricting the initial PDE to the manifold and has  smoothness of the IM (usually $C^{1+\eb}$), see \cites{FST,mik,rom-man,Zel2} and references therein. However, the standard theory of IMs requires the so-called  {\it spectral gap} assumption which looks very restrictive and is not satisfied in many physically relevant examples, see also \cites{kwean,mal-par,KZ1,K1,Zel2} for the examples where the spectral gap assumption may be relaxed using the so-called spatial averaging principle. By this reason, a lot of efforts has been made in order to improve the regularity of the inertial form without referring to the IMs, see \cites{28,Zel2} and references therein. One of the most interesting attempts (from our point of view) is the so-called Romanov theory which gives necessary and sufficient conditions for the {\it Lipschitz} continuous embeddings of the attractor to finite-dimensional spaces and allows us to construct Lipschitz continuous inertial forms, see \cites{rom-th,rom-th1}. The key application of this theory is related with the 1D system of  reaction-diffusion-advection (RDA) equations:
\begin{equation}\label{0.1}
\Dt u=\Dx u+f(u,u_x), \ \ u=(u_1,\cdots,u_m)
\end{equation}
which is considered on the interval $x\in(0,L)$ and is endowed by the proper boundary conditions. On the one hand, the spectral gap condition is not satisfied for this equation and the existence of the IM has not been proved before. On the other hand, the Romanov theory allows us to build up the Lipschitz continuous inertial form for equation \eqref{0.1} (at least in the scalar case $m=1$) under more or less general assumptions on the nonlinearity $f$, see also \cite{Kuk}. This example might be treated as an indication that a reasonable theory may be developed beyond the inertial manifolds. We mention here also an interesting attempt to extend the theory to the case of log-Lipschitz Man\'e projections and log-Lipschitz inertial forms, see \cite{28}, as well as the counterexamples to the existence of the IM in slightly more general than \eqref{0.1} classes of PDEs, see \cite{rom3}.
\par
The main aim of our study  is to give a comprehensive analysis of the existence and non-existence of inertial manifolds  for systems of 1D RDA equations. As we will see, the answer on the question about the existence of an IM strongly depends on the choice of boundary conditions.
\par
 In the present paper, we study mainly the case of Dirichlet boundary conditions where we establish the existence of IMs under more or less general assumptions on the nonlinearities. The case of Neumann boundary conditions is then reduced to the Dirichlet one using the trick with differentiation of our equations in space, see Section \ref{s4} for more details. The case of periodic boundary conditions is more delicate and is considered in the second part of our study, see \cite{KZII}. As shown there, for periodic boundary conditions we may guarantee the existence of IMs for the scalar case ($m=1$) only and will give the counterexamples which show that an IM may not exist in the vector case ($m>1$). Thus, similarly to all reasonable examples known to the authors (including the Romanov theory for scalar reaction-diffusion-advection equations), good properties of the inertial forms here are also related with the existence of an IM. This somehow confirms the conjecture stated in \cite{Zel2} that the existence of an IM gives a {\it sharp} borderline between the finite and infinite-dimensional dynamics arising in dissipative PDEs.
\par
 To be more precise, in the present paper we consider the following reaction-diffusion-advection problem:
\begin{equation}\label{1}
\Dt u+f(u)\Nx u-\Dx u+g(u)=0,\ x\in(0,L),\ \ u\big|_{x=0}=u\big|_{x=L}=0,
\end{equation}
where $u=(u_1(t,x),\cdots,u_m(t,x))$ is an unknown vector-valued function and $f$ and $g$ are given smooth functions with finite support. Thus, we have assumed from the very beginning that
the nonlinearities are already cut-off outside of the global attractor $\mathcal A$ (which is a subset of $C^1$) and do not specify more or less general  assumptions on $f$ and $g$ which guarantees global solvability, dissipativity and the existence of such global attractor, see e.g. \cites{bv1,BPZ,he,tem} and references therein for more details on this topic.
\par
 We also mention that the eigenvalues of the Laplacian in this case are $\lambda_n:=\(\frac{\pi}L\)^2n^2$ and the spectral gap condition for this equation essentially reads
\begin{equation}\label{2}
\frac{\lambda_{n+1}-\lambda_n}{\lambda_n^{1/2}+\lambda_{n+1}^{1/2}}=\frac\pi L \ge L_f,
\end{equation}
where $L_f$ is a constant related with the Lipschitz constant of the function $f$ (the Lipschitz constant of nonlinearity $g$ is not essential here, so we prefer to state the spectral gap assumption for the particular case $g=0$). Therefore, the standard methods give the existence of IM only under the assumption that the nonlinearity $f$ is small enough.
\par
The main idea of our method is to transform equation \eqref{1} (using the appropriate non-local in space change of the independent variable $u$) in such way that the obtained new equation will have {\it small} nonlinearity $f$. To this end, we set
\begin{equation}\label{3}
u(t,x)=a(t,x)v(t,x),
\end{equation}
where $a(t,x)\in GL(m)$ is a matrix depending on the solution $u$. Then, equation \eqref{1} reads
\begin{multline}\label{4}
\Dt v-\Dx v=a^{-1}[-f(av)a+2\Nx a]\Nx v+\\+[a^{-1}\Dx a -a^{-1}\Dt a -a^{-1}f(av)\Nx a ]v-a^{-1}g(av).
\end{multline}
The naive way to remove the term $\Nx v$ would be to fix the matrix $a=a(v)$ as a solution of the following ODE:
\begin{equation}\label{5}
\Nx a=\frac12 f(av)a,\ \ a\big|_{x=0}=\operatorname{Id}.
\end{equation}
However, in this case
$$
\Dx a=\frac14 f(av)^2a+\frac12 f'(av)(\Nx av+a\Nx v)a
$$
and we see that the remaining terms in the RHS of \eqref{4} will lose smoothness due to the presence of the term $\Nx v$ (the analogous thing happens also with $\Dt a$) and we end up with similar to \eqref{2} spectral gap condition for equation \eqref{4} which is again not satisfied in general. It worth mentioning that this "naive" method may work if the nonlinearity $f(u)$ is non-local and smoothing. Then, the above mentioned terms remain of order zero and the convective terms can be completely removed, see \cites{Vuk1,Vuk2} for the application of this method to the so-called Smoluchowski equation. Similar idea based on the Cole-Hopf transform has been applied in \cite{Vuk3} to the case of Burgers equation with low-wavenumber instability.
\par
Crucial observation which allows us to handle the general case of equation \eqref{1} is that we {\it need not} to remove the gradient term $\partial_x v$  in the RHS of \eqref{4} completely. Instead, as  mentioned before, it is enough to make it {\it small}
and this is possible to do in such way that $a=a(v)$ will be {\it smoothing} operator and the problem with the other terms in the RHS of \eqref{4} will not arise. Namely, as will be shown, we may modify equation \eqref{5} as follows:
\begin{equation}\label{6}
\Nx a=\frac12 f(P_K(av))a,\ \ a\big|_{x=0}=\operatorname{Id},
\end{equation}
where $P_K$ is an orthoprojector on the first $K$ eigenvalues of the Laplacian $-\Dx$ on $(0,L)$ with Dirichlet boundary conditions and $K\gg1$ (other smoothing operators are also possible). Then, the spectral gap assumption will be satisfied for the modified equation \eqref{4} and, since this equation is equivalent to the initial problem \eqref{1}, we end up with the following theorem which can be treated as the main result of the paper.

\begin{theorem}\label{TH0.main} Let the nonlinearities $f$ and $g$ be smooth and have finite supports. Then equation \eqref{1} possesses an IM in the phase space $H^1_0(0,L)$ and this manifold is $C^{1+\eb}$-smooth where $\eb>0$ is small enough.
\end{theorem}
The paper is organized as follows. The properties of the diffeomorphism generated by equation \eqref{6} are studied in Section \ref{s1}. The transformation of equation \eqref{2} to the analogous equation with respect to the new dependent variable $v$ is made in Section \ref{s2.5}. The properties of the nonlinearities involved in this equation are also studied there. The main theorem on the existence of an inertial manifold is proved in Section \ref{s2}. Finally, the cases of Neumann boundary conditions as well as  equations of the form \eqref{0.1} are considered in Section \ref{s4}.
\par
We mention also that our method of constructing the IM for systems of RDA equations is crucially based on the fact that the diffusion matrix is {\it scalar}  and does not work at least in a straightforward way  even in the case of diagonal diffusion matrix. So, the question of existence or non-existence of IMs for such systems remains open.

\section{The Auxiliary Diffeomorphism}\label{s1}
The aim of this section is to study the change of variables generated by equation \eqref{6} and formula $u=a(v)v$.
We start with the basic properties of solutions of equation \eqref{6}.

\begin{lemma}\label{Lem1} Let the above assumptions hold.  Then, for any $v\in H^1(0,L)$ and any $K\in\Bbb N$, there exists at least one solution $a:[0,L]\to GL(m,\R)$ and the following estimate holds:
\begin{equation}\label{7}
\|a\|_{W^{1,\infty}}+\|a^{-1}\|_{W^{1,\infty}}\le C,
\end{equation}
where the constant $C$ is independent of $K$ and $v$. Moreover, for sufficiently large $K\ge K_0(\|v\|_{H^1})$, the solution is unique.
\end{lemma}
\begin{proof} Let $A(x):=\frac12 f(P_K(av))$. Then, $a(x)$ is a solution matrix of the ODE
\begin{equation}\label{8}
y'(x)=A(x)y(x),\ \ y\in\R^m,
\end{equation}
i.e., every solution $y(x)$ of this equation has the form $y(x)=a(x)y(0)$. By this reason, $a(x)$ is invertible and, since the matrix  $f$ is globally bounded, by the standard estimates for the ODE \eqref{8}, we see that $a(x)$ is also uniformly bounded with respect to $x\in[0,L]$. Moreover, since the inverse matrix $b:=[a^{-1}(x)]^t$ solves the equation
\begin{equation}\label{18}
\frac d{dx}b=-\frac12 f(P_K(av))^tb,\ \ b\big|_{x=0}=\operatorname{Id},
\end{equation}
 the analogous estimate holds also for $a^{-1}(x)$. Finally, from \eqref{6} and \eqref{18}, we establish the estimate for the derivative of $a$. Thus, estimate \eqref{7} is proved and we only need to check the solvability.
\par
The existence of a solution is an immediate corollary of the fact that the operator
$$
a\mapsto \operatorname{Id}+\frac12\int_0^xf(P_K(av))a(y)\,dy
$$
is compact and continuous, say, as an operator in $L^2(0,L)$ and we have uniform a priori bounds for the solution (e.g., the Leray-Schauder degree theory can be used to verify the existence of a solution).
\par
Let us prove the uniqueness. Let $a_1$ and $a_2$ be two solutions of equation \eqref{6} and $\bar a:=a_1-a_2$. Then, this function solves
\begin{equation}\label{9}
\frac d{dx} \bar  a=\frac12f(P_K(a_1v))\bar a+\frac12 [f(P_K(a_1v))-f(P_K(a_2v))]a_2.
\end{equation}
Integrating this equation and using that $f'$ is bounded, together with estimate \eqref{7}, we get
$$
\|\bar a(x)\|\le C\int_0^x \|\bar a(y)\|+\|P_K(\bar av)(y)\|\,dy.
$$
Denoting $Q_K=\operatorname{Id}-P_K$, estimating
$$
\|\bar a(y)v(y)\|\le \|\bar a(y)\|\|v(y)\|\le C\|v\|_{H^1}\|\bar a(y)\|,
$$
and using the Gronwall inequality, we have
\begin{equation}\label{10}
\|\bar a\|_{L^\infty}^2\le C\|Q_K(\bar av)\|_{L^2}^2,
\end{equation}
where $C$ depends on $\|v\|_{H^1}$, but is independent of $K$. Moreover, inserting this estimate to \eqref{9}, we have
\begin{equation}\label{11}
\|\bar a\|_{H^1}^2\le C\|Q_K(\bar av)\|^2_{L^2}.
\end{equation}
Finally, using that
$$
\|Q_Kz\|_{L^2}^2\le \lambda_K^{-1/4}\|Q_Kz\|^2_{H^{1/4}}\le CK^{-1/2}\|z\|_{H^{1/4}},
$$
we get
\begin{equation}\label{12}
\|Q_K(\bar av)\|_{L^2}^2\le CK^{-1/2}\|\bar av\|^2_{H^{1/4}}\le CK^{-1/2}\|v\|^2_{H^1}\|\bar a\|^2_{H^1}
\end{equation}
and \eqref{11} guarantees that $\bar a=0$ if $K=K(\|v\|_{H^1})$ is large enough. Lemma \ref{Lem1} is proved.
\end{proof}
Thus, we have proved that, for every $R>0$, equation \eqref{6} defines  a map
$$
v\mapsto a(v),\ \  a:\,B(R,0,H^1)\to W^{1,\infty}(0,L;GL(m,\R))
$$
if $K\ge K_0(R)$ is large enough (here and below, we denote by $B(R,x,V)$ the ball of radius $R$ in the space $V$ centered at $x$). The next lemma gives the Lipschitz continuity of this map.

\begin{lemma}\label{Lem2} Under the assumptions of Lemma \ref{Lem1}, the map $a$ satisfies
\begin{equation}\label{13}
\|a(v_1)-a(v_2)\|_{W^{1,\infty}}+\|a^{-1}(v_1)-a^{-1}(v_2)\|_{W^{1,\infty}}\le C\|v_1-v_2\|_{H^1},
\end{equation}
for all two functions $v_1,v_2\in B(R,0,H^1)$. Moreover, the constant $C$ depends on $R$, but is independent of $K\ge K_0(R)$.
\end{lemma}
\begin{proof} Let $a_1=a_1(v_1)$ and $a_2=a_2(v_2)$ be two solutions of \eqref{6} and $\bar a=a_1-a_2$. Then, this matrix solves
\begin{equation}\label{16}
\frac d{dx}\bar a=\frac12f(P_K(a_1v_1))\bar a+\frac12 [f(P_K(a_1v_1))-f(P_K(a_2v_2))]a_2.
\end{equation}
Arguing as in the proof of Lemma \ref{Lem1}, we derive the following analogue of inequality \eqref{11}:
\begin{equation}\label{14}
\|\bar a\|_{H^1}\le C\|Q_K(\bar av_1)\|_{L^2}+C\|a_2(v_1-v_2)\|_{L^2}
\end{equation}
which, together with \eqref{12} gives
\begin{equation}\label{15}
\|\bar a\|_{H^1}\le C\|v_1-v_2\|_{L^2}
\end{equation}
and the constant $C$ is independent of $K\ge K_0$. Using now the equation \eqref{16}
together with the fact that $\|P_Kw\|_{L^\infty}\le C\|w\|_{H^1}$ where the constant $C$ is independent of $K$, we prove that
\begin{equation}\label{17}
\|a(v_1)-a(v_2)\|_{W^{1,\infty}}\le C\|v_1-v_2\|_{H^1}.
\end{equation}
The estimate for the inverse matrix $a^{-1}$ can be obtained analogously using the fact that the matrix $b(x):=[a^{-1}(x)]^t$ solves  equation \eqref{18}.
Thus, Lemma \ref{Lem2} is proved.
\end{proof}
Let us consider now the map $v\mapsto u$ given by
\begin{equation}\label{19}
u=U(v):=a(v)v.
\end{equation}
According to Lemmas \ref{Lem1} and \ref{Lem2}, this map is Lipschitz continuous as the map from $B(R,0,H^1)$ to $H^1$  if $K\ge K_0(R)$ and
\begin{equation}\label{20}
\|U(v_1)-U(v_2)\|_{H^1}\le C\|v_1-v_2\|_{H^1},
\end{equation}
where the constant $C$ depends on $R$, but is independent of $K\ge K_0$.
\par
We now describe the inverse map $v=V(u)$ which defined via $V(u)=a^{-1}(u)u$ and  $a=a(u)$ solves the linear ODE
\begin{equation}\label{21}
\frac d{dx}a=\frac12f(P_Ku)a,\ \ a\big|_{x=0}=\operatorname{Id}.
\end{equation}
Arguing analogously to Lemmas \ref{Lem1} and \ref{Lem2} (but a bit simpler since equation \eqref{21} is {\it linear}), we see that the analogues of estimates \eqref{7} and \eqref{13} hold for $a(u)$ as well (also for all $K$ independently of $R$). In addition, clearly, since the function $f$ is smooth, $a(u)$ is a $C^\infty$-map in $H^1(0,L)$. Therefore, we have proved that the inverse map $V=V(u)$ belongs to $C^\infty(H^1,H^1)$ and
\begin{equation}\label{22}
\|V(u_1)-V(u_2)\|_{H^1}\le C\|u_1-u_2\|_{H^1},\ \ u_i\in B(R,0,H^1),
\end{equation}
where the constant $C$ depends only on $R$, but is independent of $K$. Thus, we have proved the following result.

\begin{lemma}\label{Lem3} The above defined map $V: H^1(0,L)\to H^1(0,L)$ is $C^\infty$-diffeomorphism between $B(R,0,H^1)$ and $V(B(R,0,H^1))\subset H^1$
if $K\ge K_0(R)$. Moreover, the norms of $V$ and $U=V^{-1}$ as well as their derivatives are independent of $K\ge K_0(R)$.
\end{lemma}
Indeed, all assertions except of differentiability of the inverse $U=V^{-1}$ are checked above and the differentiability of $U$ can be easily derived via, say, inverse function theorem.

\begin{remark}\label{Rem1.BC} Note that, according to \eqref{19}, the maps $U$ and $V$ act not only from $H^1(0,L)$ to $H^1(0,L)$, but also from $H^1_0(0,L)$ to $H^1_0(0,L)$. In other words, the above constructed diffeomorphism  preserves the Dirichlet boundary conditions. Moreover, as not difficult to show, this diffeomorphism also preserves the regularity. In particular, it maps $H^2(0,L)$ to $H^2(0,L)$.
\end{remark}

\section{Transforming the Equation}\label{s2.5}
The aim of this section is to rewrite equation \eqref{1} in terms of the new dependent variable $v=V(u)$. To do this, we first remind the standard properties of solutions of this problem.

\begin{proposition}\label{Prop2.dis} Let $f$ and $g$ be smooth functions with finite support. Then, for every $u_0\in H^1_0(0,L)$, problem \eqref{2} possesses a unique solution
\begin{equation}\label{2.def}
 u\in C([0,T],H^1_0(0,L))\cap L^2(0,T;H^2(0,L)),\ \ T>0,
\end{equation}
satisfying $u\big|_{t=0}=u_0$. Moreover, the following dissipative estimate holds:
\begin{equation}\label{2.dis}
\|u(t)\|_{H^1}\le C\|u_0\|_{H^1}e^{-\alpha t}+C_*,
\end{equation}
where the positive constants $\alpha$, $C$ and $C_*$ are independent of $t$ and $u_0$.
\end{proposition}
\begin{proof} We give below only the derivation of the dissipative estimate \eqref{2.dis}. The rest statements are straightforward and are left to the reader, see also \cites{bv1,tem,MirZe} for more details.
\par
We first obtain the $L^2$-analogue of estimate \eqref{2.dis}. To this end, we multiply equation \eqref{1} by $u$ and integrate over $x\in(0,L)$. Then, after standard transformations, we end up with
\begin{equation}
\frac12\frac d{dt}\|u(t)\|^2_{L^2}+\|\Nx u(t)\|^2_{L^2}+(f(u)\Nx u,u)+(g(u),u)=0.
\end{equation}
Using now that $f$ and $g$ have finite support, we have
\begin{equation}
|(g(u),u)|\le C,\ \ |(f(u)\Nx u,u)|\le \frac12\|\Nx u\|^2_{L^2}+\frac12\|[f(u)]^tu\|^2_{L^2}\le C+\frac12\|\Nx u\|^2_{L^2}.
\end{equation}
Thus,
\begin{equation}
\frac d{dt}\|u(t)\|^2_{L^2}+\|\Nx u(t)\|^2_{L^2}\le C_*.
\end{equation}
Using now the Poincare inequality together with the Gronwall inequality, we derive that
\begin{equation}\label{2.l2-dis}
\|u(t)\|^2_{L^2}+\int_{t}^{t+1}\|\Nx u(s)\|^2_{L^2}\,ds\le C\|u_0\|^2_{L^2}e^{-\alpha t}+C_*
\end{equation}
and the $L^2$-analogue of the desired dissipative estimate is obtained.
\par
At the next step, we multiply equation \eqref{1} by $-\Nx^2 u$ and integrate over $x\in(0,L)$ to obtain
\begin{equation}
\frac12\frac d{dt}\|\Nx u(t)\|^2_{L^2}+\|\Nx^2 u(t)\|^2_{L^2}=(f(u)\Nx u,\Nx^2 u)+(g(u),\Nx^2 u).
\end{equation}
Using again that $f$ and $g$ are globally bounded, we arrive at
\begin{equation}
(f(u)\Nx u,\Nx^2 u)+(g(u),\Nx^2 u)\le \frac12\|\Nx^2 u(t)\|^2_{L^2}+ C(\|\Nx u(t)\|^2_{L^2}+1)
\end{equation}
and, consequently,
\begin{equation}
\frac d{dt}\|\Nx u(t)\|^2_{L^2}+\|\Nx^2 u(t)\|^2_{L^2}\le C(1+\|\Nx u(t)\|^2_{L^2}).
\end{equation}
Using again the Poincare inequality and the Gronwall inequality together with estimate \eqref{2.l2-dis}, we finally have
\begin{equation}\label{2.h1-dis}
\|\Nx u(t)\|^2_{L^2}+\int_t^{t+1}\|\Nx^2 u(s)\|^2_{L^2}\,ds\le C\|u_0\|^2_{H^1}e^{-\alpha t}+C_*.
\end{equation}
Thus, the desired dissipative estimate is verified and the proposition is proved.
\end{proof}
According to the proved proposition, equation \eqref{1} generates a solution semigroup $S(t)$ in the phase space $\Phi:=H_0^1(0,L)$ via
\begin{equation}\label{2.sem}
S(t):\Phi\to\Phi,\ \ S(t)u_0:=u(t).
\end{equation}
Moreover, according to \eqref{2.dis}, this semigroup is dissipative and possesses an absorbing ball
\begin{equation}\label{2.abs}
\mathbb B:=\{u_0\in\Phi,\ \ \|u_0\|_{\Phi}\le R/2\}
\end{equation}
if $R$ is large enough. Our next task is to establish the existence of a global attractor for this semigroup.
\begin{definition}\label{Def2.attr} Recall that a set $\mathcal A$ is a global attractor of the semigroup $S(t)$ in $\Phi$ if the following conditions are satisfied:
\par
1. The set $\mathcal A$ is compact in $\Phi$.
\par
2. The set $\mathcal A$ is strictly invariant, i.e., $S(t)\mathcal A=\mathcal A$ for all $t\ge0$.
\par
3. It attracts the images of all bounded sets as time tends to infinity, i.e., for every bounded set $B\subset\Phi$ and every neighbourhood $\mathcal O(\mathcal A)$ of the attractor $\mathcal A$, there exists time $T=T(\mathcal O,B)$ such that
$$
S(t)B\subset\mathcal O(\mathcal A)
$$
for all $t\ge T$.
\end{definition}

\begin{proposition}\label{Prop2.attr} Let the above assumptions hold. Then, the solutions semigroup $S(t)$ generated by equation \eqref{1} possesses a global attractor $\mathcal A$ in the phase space $\Phi$. Moreover, this attractor is bounded in $H^2(0,L)$ and is generated by all complete bounded solutions of equation \eqref{1}:
\begin{equation}\label{2.rep}
\mathcal A=\mathcal K\big|_{t=0},
\end{equation}
where $\mathcal K\subset C_b(\R,\Phi)$ consists of all solutions $u(t)$ of problem \eqref{1} which are defined for all $t\in\R$ and bounded.
\end{proposition}
\begin{proof} Indeed, according to the abstract theorem on the attractors existence (see e.g., \cite{bv1}), we need to check that the maps $S(t):\Phi\to\Phi$ are continuous for every fixed time $t$ and that the semigroup possesses a compact absorbing set. The continuity is obvious in our case. Moreover, the ball \eqref{2.abs} is the absorbing set for the solution semigroup if $R$ is large enough. However, this ball is not compact in $\Phi$. In order to overcome this difficulty, it is enough to note that the set $S(1)\mathbb B$ is also absorbing and that, due to the parabolic smoothing property, $S(1)\mathbb B$ is bounded in $H^2(0,L)$ and, by this reason, is compact in $\Phi$. Thus, all of the conditions of the above mentioned abstract theorem are verified and the existence of a global attractor is also verified. Formula \eqref{2.rep} is a standard corollary of this theorem and the fact that $\mathcal A$ is bounded in $H^2$ follows from the fact that the attractor is always a subset of the absorbing set. Thus, the proposition is proved.
\end{proof}

\begin{remark}\label{Rem2.cut} Recall that the assumptions that $f$ and $g$ have finite support are
not physically relevant. More realistic would be to assume, for instance, that $f$ and $g$ are polynomials in $u$ (e.g., $f(u)=u$ in the case of Burgers equation). However, verification of the key dissipative estimate \eqref{2.dis} is much more difficult in this case and requires extra assumptions especially in the case of systems, see the discussion in \cite{BPZ} for the case of coupled Burgers equations. On the other hand, since any trajectory enters the absorbing ball in finite time and never leaves it, the solutions outside of this ball are not important to study the long-time behavior, so the nonlinearities may be cut off outside of the absorbing ball and this reduces the general case to the case of nonlinearities with finite support. Since this trick is standard for the theory of inertial manifolds, we do not discuss here general assumptions on $f$ and $g$ which guarantee the validity of the dissipative estimate \eqref{2.dis}, assuming instead that the cut off procedure is already done, and start from the very beginning with the nonlinearities with finite support.
\end{remark}

Our next task is to transform equation \eqref{1} to the analogous equation with respect to the new dependent variable $v=V(u)$. To this end, we fix the radius $R$ in such a way that the set $\mathbb B$ is an absorbing ball for the semigroup $S(t)$ and
also fix $K_0=K_0(R)$ as in Lemma \ref{Lem3}. Then, the map $V$ is a $C^\infty$-diffeomorphism in the neighborhood $B(R,0,\Phi)$ of the attractor $\Cal A$ and, therefore, the change of variables $v=V(u)$ is well-defined and one-to-one in this neighborhood. We now study equation \eqref{4} which due to \eqref{6} has the form
\begin{equation}\label{23}
\Dt v-\Dx v=F_1(v)\Nx v+F_2(v),
\end{equation}
where
\begin{equation}\label{24}
F_1(v)=a^{-1}(v)[f(P_K(a(v)v))-f(a(v)v)]a(v)
\end{equation}
and
\begin{equation}\label{25}
F_2(v):=[a^{-1}\Dx a -a^{-1}\Dt a -a^{-1}f(av)\Nx a ]v-a^{-1}g(av).
\end{equation}
Recall also that the map $a$ as well as $F_i$ depend on a parameter $K\ge K_0$. We start with the most complicated operator $F_2$.
\begin{lemma}\label{Lem4} For sufficiently large $K\ge K_0$ the map $F_2$ belongs to $C^\infty(B(r,0,H^1))$, where $r$ is chosen in such way that
$$
V(B(R,0,H^1))\subset B(r,0,H^1)
$$
and, in particular,
\begin{equation}\label{26}
\|F_2(v_1)-F_2(v_2)\|_{H^1}\le C_K\|v_1-v_2\|_{H^1_0},\ \ v_i\in B(r,0,H^1_0),
\end{equation}
where the constant $C_K$ depends on $K$.
\end{lemma}
\begin{proof} Actually, definition \eqref{25} of the operator $F_2$ involves two non-trivial terms $\Dx a(v)$ and $\Dt a(v)$, the other terms can be estimated in a straightforward way using Lemmas \ref{Lem1} and \ref{Lem2}. Let us first treat $\Dx a$. Differentiating equation \eqref{6} in $x$, we get
\begin{equation}\label{27}
\frac {d^2}{dx^2}a=\frac12 f(P_K(av))\Nx a+\frac12 f'(P_K(av))(\Nx P_K(av))a.
\end{equation}
The estimate for the $L^\infty$-norms of the terms $a(v_1)-a(v_2)$ as well as for the terms $\Nx a(v_1)-\Nx a(v_2)$ are obtained in Lemma \ref{Lem2}. Furthermore, since
\begin{equation}\label{28}
P_K(av)=\sum_{k=1}^K(av,e_k)e_k
\end{equation}
and all $e_k$ are smooth, we also have
$$
\|P_K(a(v_1)v_1-a(v_2)v_2)\|_{L^\infty}\le C_K(\|v_1-v_2\|_{L^2}+\|a(v_1)-a(v_2)\|_{L^2})
$$
and, therefore,
$$
\|\Dx a(v_1)-\Dx a(v_2)\|_{L^\infty}\le C_K\|v_1-v_2\|_{H^1}.
$$
Moreover, differentiating equation \eqref{27} once more in $x$ and arguing analogously, we derive that
\begin{equation}\label{29}
\|\Nx^3 a(v_1)-\Nx^3 a(v_2)\|_{L^\infty}\le C_K\|v_1-v_2\|_{H^1}
\end{equation}
which is sufficient for estimating the $H^1$ norm of $F_2(v_1)-F_2(v_2)$.
\par
Let us now treat the term $\Dt a(v)$. Being pedantic, this term is even not defined yet since we need to use equation \eqref{1} or \eqref{4} in order to evaluate the time derivative. To define it, we differentiate equation \eqref{6} in time and write
\begin{equation}\label{30}
\frac d{dx}\Dt a=\frac12 f(P_Ku)\Dt a+\frac12 f'(P_Ku)(P_K\Dt u)a,\ \ \Dt a\big|_{x=0}=0,
\end{equation}
where $u=U(v)$. Then, using equation \eqref{1}, we write
\begin{equation}\label{31}
P_K\Dt u=\sum_{k=1}^K[(u,\Dx e_k)-(f(u)\Nx u,e_k)-(g(u),e_k)]e_k.
\end{equation}
Thus, we {\it define} operator $\Dt a(v)$ as a solution of the ODE \eqref{30} where $u=U(v)$ and $P_K\Dt u$ is defined by \eqref{31}. Moreover, using \eqref{20}, we see that
\begin{equation}\label{32}
\|P_K\Dt u(v_1)-P_K\Dt u(v_2)\|_{L^\infty}\le C_K\|v_1-v_2\|_{H^1}.
\end{equation}
Then, analogously to the proof of Lemma \ref{Lem2}, we deduce from equation \eqref{30} that
\begin{equation}\label{33}
\|\Dt a(v_1)-\Dt a(v_2)\|_{W^{1,\infty}}\le C_K\|v_1-v_2\|_{H^1}.
\end{equation}
Thus, the second non-trivial term $\Dt a$ is also treated and \eqref{26} follows in a straightforward way from Lemma \ref{Lem2} and estimates \eqref{29} and \eqref{33}. Lemma \ref{Lem4} is proved.
\end{proof}
We now return to the operator $F_1(v)$.
\begin{lemma}\label{Lem5} Under the above assumptions the operator $F_1(v)$ satisfies the following estimates:
\begin{equation}\label{34}
\|F_1(v)\|_{L^\infty}\le CK^{-1/2}
\end{equation}
and
\begin{equation}\label{35}
\|F_1(v_1)-F_1(v_2)\|_{L^\infty}\le CK^{-1/2}\|v_1-v_2\|_{H^1_0},
\end{equation}
where $v,v_1,v_2\in B(r,0,H^1_0)$ and the constant $C$ depends on $r$, but is independent of $K$.
\end{lemma}
\begin{proof} Indeed, due to Lemma \ref{Lem2}, it is sufficient to verify the above estimates for the operator
\begin{equation}\label{36}
F_0(v)=f(av)-f(P_K(av))=\int_0^1f'(s av+(1-s)P_K(av))\,ds\, Q_K(av).
\end{equation}
To this end, we use that
\begin{equation}\label{37}
\|Q_K(av)\|_{L^\infty}\le C\|Q_K(av)\|_{L^2}^{1/2}\|Q_K(av)\|_{H^1_0}^{1/2}\le CK^{-1/2}\|av\|_{H^1_0}\le CK^{-1/2}
\end{equation}
and, analogously,
\begin{equation}\label{38}
\|Q_K(a(v_1)v_1-a(v_2)v_2)\|_{L^\infty}\le CK^{-1/2}\|v_1-v_2\|_{H^1_0}.
\end{equation}
Estimates \eqref{37} and \eqref{38} together with Lemma \ref{Lem2} allow us to deduce \eqref{34} and \eqref{35} as an elementary calculation. Lemma \ref{Lem5} is proved.
\end{proof}
\begin{remark}\label{Rem3.g0} Note that in general $F_2$ does not preserve the Dirichlet boundary conditions since
$$
F_2(v)\big|_{x=0,L}=a^{-1}(v)\big|_{x=0,L}g(0)\ne0.
$$
However, it will map $H^1_0$ to $H^1_0$ if we assume in addition that
\begin{equation}\label{3.g0}
g(0)=0.
\end{equation}
\end{remark}
\section{The Inertial Manifold}\label{s2}

We are now ready to construct the desired inertial manifold for equation \eqref{1}. As shown in the previous section, this equation is equivalent to \eqref{23} at least in the neighbourhood of the absorbing set $\mathbb B$. By this reason, we may construct the inertial manifold for equation \eqref{23} instead. However, the nonlinearities $F_1$ and $F_2$ in this equation are still not globally defined on $\Phi$. To overcome this problem, we need, as usual,
to cut-off the nonlinearities outside of a large ball  making them globally Lipschitz continuous.
\par
Namely, we introduce a smooth  cut-off function $\varphi (z)$ such that $\varphi(z)\equiv 1$
for $z\le r_1^2$ and $\varphi(z)=0$, $z\ge r^2$, where $r_1$ is such that $V(\mathbb B)\subset B(r_1,0,H^1_0)$ and $r>r_1$ and $K_0=K_0(r)$ are chosen in such way that the inverse map $U=U(v)$ is a diffeomorphism on $B(r,0,H^1_0)$ and the assertions of Lemma \ref{Lem3} hold for every $K\ge K_0$.
\par
Finally, we modify equation \eqref{23} as follows:
\begin{equation}\label{39}
\Dt v-\Dx v=\varphi(\|v\|^2_{H^1})F_1(v)\Nx v+\varphi(\|v\|^2_{H^1})F_2(v):=\Cal F_1(v)+\Cal F_2(v).
\end{equation}
Then, according to Lemma \ref{Lem4} and \ref{Lem5}, we have
\begin{equation}\label{40}
\|\Cal F_1(v_1)-\Cal F_1(v_2)\|_{L^2}\le CK^{-1/2}\|v_1-v_2\|_{H^1_0}
\end{equation}
and
\begin{equation}\label{41}
\|\Cal F_2(v_1)-\Cal F_2(v_2)\|_{H^1}\le C_K\|v_1-v_2\|_{H^1_0},
\end{equation}
where $v_i\in H^1_0(0,L)$ are arbitrary and the constant $C$ is independent of $K$.
Moreover, under the additional technical assumption \eqref{3.g0}, the map $\mathcal F_2$ will act from $H_0^1$ to $H^1_0$.
\par
Thus, instead of constructing the inertial manifold for the initial equation \eqref{1}, we will construct it for the transformed problem \eqref{39}. For the convenience of the reader, we recall the definition of the inertial manifold for this equation.

\begin{definition}\label{Def3.IM} A finite-dimensional submanifold $\mathcal M$ of the phase space $\Phi$ is an inertial manifold for problem \eqref{39} if the following conditions are satisfied:
\par
1. The manifold $\mathcal M$ is strictly invariant with respect to the solution semigroup $\tilde S(t)$ of equation \eqref{39}, i.e., $\tilde S(t)\mathcal M=\mathcal M$ for all $t\ge0$.
\par
2. The manifold $\mathcal M$ is a graph of a Lipschitz continuous function $M: P_n\Phi\to Q_n\Phi$ for some $n\in\mathbb N$. Here and below we denote by $P_n$ the orthoprojector in $\Phi$ to the first $n$-eigenvalues of the Laplacian and $Q_n=\operatorname{Id}-P_n$.
\par
3. The manifold $\mathcal M$ possesses the so-called exponential tracking property, i.e., for every trajectory $v(t)$, $t\ge0$, of problem \eqref{39} there is a trajectory $\tilde v(t)$ belonging to $\mathcal M$ such that
\begin{equation}\label{3.phase}
\|v(t)-\tilde v(t)\|_{\Phi}\le C\|v(0)-\tilde v(0)\|_{\Phi}e^{-\theta t}
\end{equation}
for some positive $C$ and $\theta$.
\end{definition}
We are now ready to state and prove the main result of the paper.

\begin{theorem}\label{Th1} Under the above assumptions on $f$ and $g$, equation \eqref{39} possesses an inertial manifold $\mathcal M\subset\Phi$ of smoothness $C^{1+\eb}$, for some $\eb>0$.
\end{theorem}
\begin{proof} We first explain the main idea of constructing the inertial manifold restricting ourselves to the special case when $g(0)=0$, see the end of this section for the explanations on how to remove this technical assumption. As known, in order to do so, we need to verify the so-called spectral gap conditions, see \cites{FST,Zel2} and references therein. Indeed, the nonlinearity $\mathcal F_1$ decreases the smoothness by one, so the spectral gap condition for it reads
\begin{equation}\label{4.f1}
\frac{\lambda_{n+1}-\lambda_n}{\lambda_{n+1}^{1/2}+\lambda_{n}^{1/2}}>L_1,
\end{equation}
where $L_1$ is the Lipschitz constant of $\mathcal F_1$ as the map from $H^1_0$ to $L^2$. On the other hand, the nonlinearity $\mathcal F_2$ is globally bounded in $H^1_0$, so the spectral gap condition for it reads
\begin{equation}\label{4.f2}
\lambda_{n+1}-\lambda_n>2L_2,
\end{equation}
where $L_2$ is the Lipschitz constant of $\mathcal F_2$ as the map from $H^1_0$ to $H^1_0$.
\par
In our case, we have an extra parameter $K\in\mathbb N$ involved and
\begin{equation}
L_1=CK^{-1/2},\ \ \ L_2=C_K.
\end{equation}
Moreover, the eigenvalues of the Laplacian
\begin{equation}\label{42}
\lambda_n=\frac{\pi^2}{L^2}n^2\ \text{and}\ \ \frac{\lambda_{n+1}-\lambda_n}{\lambda_{n+1}^{1/2}+\lambda_{n}^{1/2}}=\frac\pi L.
\end{equation}
Thus, fixing $K$ being large enough, we may make the Lipschitz constant $L_1=CK^{-1/2}$ of the nonlinearity $\Cal F_1$ small enough to satisfy the spectral gap condition \eqref{4.f1}. Then, since
\begin{equation}\label{43}
\lambda_{n+1}-\lambda_{n}=\frac{\pi^2}{L^2}(2n+1),
\end{equation}
we may find $n=n(K)$ large enough so that the Lipschitz constant $L_2=C_K$ of the second nonlinearity $\Cal F_2(v)$ satisfies the spectral gap condition \eqref{4.f2}. Thus, the spectral gap conditions are satisfied and the IM for equation \eqref{39} exists. However, the standard theory works with only one nonlinearity ($\mathcal F_1$ or $\mathcal F_2$) and its validity for the case where both nonlinearities are simultaneously present in the equation should be verified/explained. By this reason, we briefly recall  below the proof of the inertial manifold existence and show that the slight modification of assumptions \eqref{4.f1} and \eqref{4.f2} works indeed for the case where both nonlinearities are involved simultaneously.
\par
Following the Perron method, the desired manifold is found by solving the backward in time boundary value problem
\begin{equation}\label{4.back}
\partial_t v-\Dx v=\mathcal F_1(v)+\mathcal F_2(v),\ \ \ P_nv\big|_{t=0}=v_0,\ \ t\le0
\end{equation}
in the weighted space $L^2_{e^{-\theta t}}(\mathbb R_-,\Phi)$ with $\theta:=\frac{\lambda_{n+1}+\lambda_n}2$. The solution of this equation is usually constructed by Banach contraction theorem and the desired map $M: P_n\Phi\to Q_n\Phi$ is then defined via
\begin{equation}\label{4.M}
M(v_0):=Q_nv(0),
\end{equation}
see \cite{Zel2} for the details. To apply the Banach contraction theorem, we introduce the function $w=w(v_0)$ as a solution of the linear problem
\begin{equation}\label{4.lin}
\Dt w-\Dx w=0,\ \ P_nw\big|_{t=0}=v_0.
\end{equation}
Then, as not difficult to see that this problem is uniquely solvable in
$L^2_{e^{-\theta t}}(\R_-,\Phi)$, so the associated linear operator
$$
w:\,P_n\Phi\to L^2_{e^{-\theta t}}(\R_-,\Phi)
$$
is well-defined. Introducing the function $z=v-w$, we transform \eqref{4.back} to
\begin{equation}\label{4.back1}
\partial_t z-\Dx z=\mathcal F_1(z+w)+\mathcal F_2(z+w),\ \ \ P_n z\big|_{t=0}=0,\ \ t\le0.
\end{equation}
Furthermore, as also not difficult to show, the linear non-homogeneous problem
\begin{equation}\label{4.hom}
\Dt z-\Dx z=h(t),\ \ t\in\R,
\end{equation}
is uniquely solvable in the space $L^2_{e^{-\theta t}}(\R,\Phi)$ for any $h\in L^2_{e^{-\theta t}}(\R,\Phi)$, so the linear operator
$$
\mathcal R:\,L^2_{e^{-\theta t}}(\R,\Phi)\to L^2_{e^{-\theta t}}(\R,\Phi)
$$
is well defined. Using this operator, problem \eqref{4.back1} can be rewritten in the equivalent form as follows:
\begin{equation}\label{4.back2}
z=\mathcal R(\chi_-\mathcal F_1(z+w)+\chi_-\mathcal F_2(z+w)),
\end{equation}
where $\chi_-(t)=0$ for $t\ge0$ and $\chi_-(t)=1$ for $t<0$, see \cite{Zel2} for the details.
\par
To estimate the Lipschitz constant of the right-hand side of \eqref{4.back2}, we need the following lemma.
\begin{lemma}\label{Lem4.R} Under the above assumptions the following estimates for the norms of $\mathcal R$ hold:
\begin{equation}\label{4.norm1}
\|\mathcal R\|_{\mathcal L(L^2_{e^{-\theta t}}(\R,\Phi),L^2_{e^{-\theta t}}(\R,\Phi))}\le \frac2{\lambda_{n+1}-\lambda_n}
\end{equation}
and
\begin{equation}\label{4.norm2}
\|\mathcal R\|_{\mathcal L(L^2_{e^{-\theta t}}(\R,L^2),L^2_{e^{-\theta t}}(\R,\Phi))}\le \frac{2\lambda_{n+1}^{1/2}}{\lambda_{n+1}-\lambda_n}.
\end{equation}
\end{lemma}
Indeed, these estimates are the straightforward corollaries of the key estimate
$$
\|y\|_{L^2_{e^{-\theta t}}(\R)}^2\le\frac1{(\lambda_k-\theta)^2}\|h\|^2_{L^2_{e^{-\theta t}}(\R)}
$$
for the solution of the 1st order ODE
$$
\frac d{dt}y+\lambda_k y=h(t),
$$
see \cite{Zel2} for the details.
\par
Thus, the Lipschitz constant of the right-hand side of \eqref{4.back2} can be estimated by
\begin{equation}\label{4.lip}
{\rm Lip}\le \frac{2\lambda_{n+1}^{1/2}L_1}{\lambda_{n+1}-\lambda_n}+\frac{2L_2}{\lambda_{n+1}-\lambda_n}
\end{equation}
and this constant is indeed less than one (which allows us to apply the Banach contraction theorem) if
\begin{equation}\label{4.f12}
\frac{\lambda_{n+1}-\lambda_n}{\lambda_{n+1}^{1/2}}>4L_1,\ \ \lambda_{n+1}-\lambda_n>4L_2.
\end{equation}
Therefore, since $L_1=CK^{-1/2}$, we may fix $K$ to be large enough to satisfy the first condition of \eqref{4.f12} for all $n\in\mathbb N$. Then, since $L_2=C_K$ and $\lambda_{n+1}-\lambda_n\sim C n$, we always may find $n$ in such a way that the second condition of \eqref{4.f12} will be also satisfied. Thus the desired inertial manifold could be indeed constructed by the Banach contraction theorem. As shown in \cite{Zel2} both the exponential tracking and $C^{1+\eb}$-regularity of this manifold are also the straightforward corollaries of this contraction and the theorem is proved.
\end{proof}
\begin{remark}\label{Rem3.gen} Recall that we have proved  the main Theorem \ref{Th1} on inertial manifold existence for equations \eqref{1} under the additional assumption that $g(0)=0$. This assumption  is posed only in order to have $F_2(v)\in H^1_0(0,L)$ if $v\in H_0^1(0,L)$ and can be easily removed. Indeed, in the general case, the boundary conditions are not preserved, but $F_2$ still maps $H^1_0$ to $H^1$. Using that $H^s_0=H^s$ for $s<1/2$, we may treat the nonlinearity $F_2$ as the Lipschitz map from the phase space $\Phi=H^1_0(0,L)$ to, say, $H^{1/4}(0,L)=H^{1/4}_0(0,L)=D((-\Dx)^{-3/8}\Phi)$. The spectral gap condition for such nonlinearities reads
$$
Cn^{1/4}\sim\frac{\lambda_{n+1}-\lambda_n}{\lambda_n^{3/8}+\lambda_{n+1}^{3/8}}>4L_2
$$
and still can be satisfied by choosing $n=n(K,g)$ large enough.
\end{remark}
\begin{remark}\label{Rem3.F} Mention that the Perron method used in the proof of Theorem \ref{Th1} automatically gives the so-called absolutely normally hyperbolic inertial manifolds if the nonlinearities are smooth enough, so the constructed inertial manifold for the reaction-diffusion-advection problem is absolutely normally hyperbolic, see e.g., \cites{RTem,rom-16} for the definition and more details. As we will see in the second part of our study (see \cite{KZII}), this property may disappear in the case of periodic boundary conditions.
\par
Recall also that the standard definition of an inertial manifold, see Definition \ref{Def3.IM}, usually assumes that the IM can be presented as a graph  over the linear subspace generated by the lower Fourier modes. In our case, it is so for the transformed equation \eqref{39}. However, if we return back to equation \eqref{6}, the associated invariant manifold $U(\mathcal M)$ a priori does not have this structure since it is not necessarily can be nicely projected to the linear subspace generated by the lower Fourier modes of this equation. On the other hand, as follows from the Romanov theory, see \cites{rom-th,rom-th1} for the details, the attractor $\mathcal A$ of equation \eqref{6} can be projected in a bi-Lipschitz way to the plane $P_n\Phi$ if $n$ is large enough. Since the manifold $U(\mathcal M)$ is smooth and contains the attractor, we may expect that  a sufficiently small neighbourhood of the attractor in $U(\mathcal M)$ can be nicely projected to $P_n\Phi$. By this reason, we may  a posteriori expect that the constructed inertial manifold is still a graph over the lower Fourier modes. We will return to this problem somewhere else.
\end{remark}

\section{Neumann boundary conditions and more general equations}\label{s4}
In this concluding section, we discuss possible extensions of the obtained result to more general RDA systems and other boundary conditions. We start with the case of Neumann boundary conditions (recall that the case of periodic boundary conditions is postponed to the second part of our study, see \cite{KZII}).
To be more precise, let us consider system of equations \eqref{1} with  the Neumann boundary conditions:
\begin{equation}\label{4.Neu}
\Dt u+f(u)\Nx u+g(u)=\Dx u-u,\ \
\partial_x u\big|_{x=0}=\partial_x u\big|_{x=L}=0
\end{equation}
assuming as before that $u$ is vector-valued: $u=(u_1,\cdots, u_m)$ and the nonlinearities $f$ and $g$ are of class $C_0^\infty$ (we put the extra term $-u$ in the right-hand side in order to restore the dissipativity which may be lost otherwise due to the lack of the Friedrichs inequality). Then, exactly as in the case of Dirichlet boundary conditions, problem \eqref{4.Neu} generates a dissipative solution semigroup $S(t)$ in the phase space $\Phi:=[H^1(0,L)]^m$ and possesses a global attractor $\mathcal A$ in this phase space. In particular, we have a bounded invariant absorbing set
$\mathcal B\subset \Phi$:
\begin{equation}\label{4.abs}
S(t)\mathcal B\subset\mathcal B,\ \ \mathcal O(\mathcal A)\subset\mathcal B,
\end{equation}
where $\mathcal O$ is a neighbourhood of $\mathcal A$ in $\Phi$. However, we cannot directly apply the change of variables $u=av$ since it will not preserve the Neumann boundary conditions. Indeed, instead, we will end up with the nonlinear boundary conditions
\begin{equation}
0=\partial_x u\big|_{x=0,L}=\partial_x av\big|_{x=0,L}+a\partial_xv\big|_{x=0,L}
\end{equation}
and $\partial_xa\big|_{x=0,L}\ne0$ at least for our construction. Thus, this construction should at least be essentially modified and the possibility of the proper modification looks doubtful taking in mind the counterexamples constructed in the second part of our study for the case of Dirichlet-Neumann boundary conditions. By this reason, we will proceed in an alternative way embedding equations \eqref{4.Neu} to a larger system of RDA equations with a special structure which allows us to apply the change of variables exactly as in the case of Dirichlet boundary conditions.
\par
Namely, differentiating equation \eqref{4.Neu} in $x$ and denoting $w=\Nx u$, we end up with the following system of RDA equations:
\begin{equation}\label{4.N}
\begin{cases}
\Dt u+f(u)w+g(u)=\Dx u-u,\ \  \partial_xu\big|_{x=0,L}=0,\\ \Dt w+f'(u)[w,w]+g'(u)w+f(u)\Nx w=\Dx w-w,\ \  w\big|_{x=0,L}=0.
\end{cases}
\end{equation}
This is a system of $2m$ RDA equations of the form \eqref{1} endowed by Dirichlet-Neumann boundary conditions and the initial system \eqref{4.Neu} is embedded as a restriction to the infinite-dimensional invariant submanifold given by the condition $w=\Nx u$. The corresponding embedding map is obviously defined by
\begin{equation}\label{4.emb}
\bar E(u):= (u,\Nx u).
\end{equation}
Moreover, applying the proper cut-off to the term $f'(u)[w,w]$, we may assume without loss of generality that the nonlinearities in \eqref{4.N} are also of class $C_0^\infty$, so we have the global well-posedness and dissipativity of the solution semigroup $\bar S(t)$ in the phase space
$$
\Psi:=[H^1(0,L)]^m\times[H^1_0(0,L)]^m.
$$
On the other hand, since \eqref{4.N} does not contain the terms $\Nx u$ in the first $m$ equations any more, we need not to change the $u$-component and  may transform only the second component $w$ via $w=a(u)v$, where $a(u)$ solves
$$
\frac d{dx} a=\frac12 f(P_Ku)a,\ \ a\big|_{x=0}=\operatorname{Id}
$$
and, in contrast to the previous theory, $P_K$ is an orthoprojector to first $K$ eigenvectors of the Laplacian with {\it Neumann} boundary conditions (this is necessary in order to get the smallness of the operator $F_1$). Then, since $w$ has Dirichlet boundary conditions which are preserved under the transform,  repeating word by word the above arguments, we obtain the existence of an inertial manifold for problem \eqref{4.N}.
\par
Note again that in contrast to the case of Dirichlet boundary conditions, we now embed the considered equation \eqref{4.Neu} into a larger system of equations and construct the IM for this larger system only (similarly to the unsuccessful attempt to build up the IMs for 2D Navier-Stokes equations, see \cites{KZkwak,Kwak,rom-16}). Being pedantic, we need to extend properly the definition of an IM in order to make the statements rigorous.

\begin{definition}\label{Def.IMext} Let $S(t):\Phi\to\Phi$ be a semigroup acting in a Banach space $\Phi$ and possessing the invariant bounded absorbing set $\mathcal B$ in it. Assume that
\par
1) There exists another Banach space $\Psi$ and a dissipative semigroup $\bar S(t)$ in $\Psi$
\par
2) There exists a bi-Lipschitz embedding $E:\mathcal B\to\Psi$ such that
\begin{equation}\label{4.iso}
\bar S(t)=E\circ S(t)\circ E^{-1}
\end{equation}
on $E(\mathcal B)\subset\Psi$.
\par
3) The dynamical system $\bar S(t)$ possesses an IM in the phase space $\Psi$ in the sense of Definition~\ref{Def3.IM}.
\par
Then $\mathcal M$ is referred as a (generalized) inertial manifold for the semigroup $S(t)$. This manifold is called $C^{1+\eb}$-smooth if {\it both} $E$ and $\mathcal M$ are $C^{1+\eb}$-smooth.
\end{definition}
Then, we have proved the following theorem.
\begin{theorem}\label{Th4.Neu} Let the functions $f$ and $g$ be of class $C_0^\infty$. Then, problem \eqref{4.Neu} possesses an IM in the sense of Definition \ref{Def.IMext} and this manifold is $C^{1+\eb}$-smooth.
\end{theorem}
\begin{remark}\label{Rem4.IM} Note that the above given definition is useful even in the standard situation when only the cut-off procedure is used. Indeed, in this case $\Phi=\Psi$, $E$ is an identity map and $\bar S(t)$ is a dynamical semigroup generated by the equation with cut-off nonlinearities. Even in this case, the usage of Definition \ref{Def.IMext} allows us to define rigorously an IM for the {\it initial}, not truncated system avoiding introducing the artificial boundaries on the manifolds and preserving its strict invariance (with respect to the truncated semigroup).
\par
In the case of RDA equations with Dirichlet boundary conditions, again $\Phi=\Psi$, but now $E$ is no more an identity map, but a diffeomorphism $U$ constructed in Section \ref{s1} and the semigroup $\bar S(t)$ coincides with the solution semigroup $\tilde S(t)$ of equation \eqref{39}.
\par
Finally, in the case of Neumann boundary conditions, $\Phi=[H^1(0,L)]^m$, $\Psi=\Phi\times [H^1_0(0,L)]^m$ and the map $E$ is actually a composition of the embedding map \eqref{4.emb} and a diffeomorphism $U$ acting on the Dirichlet component $w$ of system \eqref{4.N}. We also note that in this case, we have only that $E(\mathcal A)\subset \mathcal M$, but the trajectories on $\mathcal M$ which lie outside of the attractor are not generated by the trajectories of the initial problem \eqref{4.Neu} no matter how close to the attractor they are. We however think that this drawback is not essential since it destroys neither the exponential tracking property nor the existence of smooth finite-dimensional inertial forms, so we did not put any efforts to overcome it.
\end{remark}
The second part of this section is devoted to the case of general nonlinear dependence of the nonlinearity $f$ on $\Nx u$. We start with more simple case of Neumann boundary conditions:
\begin{equation}\label{4.fN}
\Dt u+f(u,\Nx u)=\Dx u-u,\ \ \Nx u\big|_{x=0}=\Nx u\big|_{x=L}=0,
\end{equation}
where $u=(u_1,\cdots,u_m)$ is vector-valued and $f$ is a given nonlinearity of class $C_0^\infty$. Differentiating this equation by $x$ and denoting $w=\Nx u$, we end up with
\begin{equation}\label{4.fND}
\begin{cases}
\Dt u+f(u,w)=\Dx u-u,\ \ \Nx u\big|_{x=0}=\Nx u\big|_{x=L}=0,\\
\Dt w+f'_u(u,w)w+f'_{w}(u,w)\Nx w=\Dx w-w,\ \ w\big|_{x=0}=w\big|_{x=L}=0
\end{cases}
\end{equation}
This equation has the structure of the RDA system \eqref{4.N}, therefore arguing exactly as before, we get the following result.

\begin{theorem}\label{Th4.fN} Let the nonlinearity $f$ be of class $C_0^\infty$. Then, equation \eqref{4.fN} possesses an IM in the phase space $\Phi:=[H^1(0,L)]^m$ in the sense of Definition \ref{Def.IMext}.
\end{theorem}
We now turn to the case of Dirichlet boundary conditions. We consider the problem
\begin{equation}\label{4.fD}
\Dt u+f(u,\Nx u)=\Dx u-u,\ \ u\big|_{x=0}=u\big|_{x=L}=0
\end{equation}
and start with the simplifying assumption that
\begin{equation}\label{4.simple}
f(0,\Nx u)\equiv0.
\end{equation}
In this case, from equation \eqref{4.fD} we conclude that $\Dx u=0$ at $x=0$ and $x=L$ and, therefore, the function $w=\Nx u$ satisfies the Neumann boundary conditions. Then, we may differentiate the equation by $x$ once more and denoting $\theta=\Nx w$ end up with the following system of RDA equations:
\begin{equation}\label{4.rred}
\begin{cases}
\Dt u+f(u,w)=\Dx u-u,\ \ u\big|_{x=0,L}=0,\\
\Dt w+f'_u(u,w)w+f'_{\Nx u}(u,w)\theta=\Dx w-w,\ \ \partial_xw\big|_{x=0,L}=0,\\
\Dt\theta+f''_{uu}(u,w)[w,w]+2f''_{u,\Nx u}(u,w)[w,\theta]+\\+
f''_{\Nx u,\Nx u}(u,w)[\theta,\theta]+f'_u(u,w)\theta+f'_{\Nx u}(u,w)\Nx\theta=\Dx \theta-\theta,\ \ \theta\big|_{x=0,L}=0.
\end{cases}
\end{equation}
These equations have the form of \eqref{4.N} and only the third equation contains a convective term, so we need to change only the $\theta$-variable where the boundary conditions are the Dirichlet ones. Therefore, repeating word by word the above arguments, we obtain the following result.

\begin{theorem}\label{Th4.fD1} Let the nonlinearity $f$ be of class $C_0^\infty$ and satisfy the extra assumption \eqref{4.simple}. Then the system \eqref{4.fD} of RDA equations possesses an IM in the sense of Definition \ref{Def.IMext}.
\end{theorem}
We conclude our exposition by considering the general case of \eqref{4.fD} when the simplifying assumption \eqref{4.simple} is not satisfied. In this case, the function $w=\Nx u$ will satisfy the nonlinear boundary conditions
$$
\Nx w\big|_{x=0,L}=f(0,w)\big|_{x=0,L},
$$
so the differentiation in $x$ does not  help much and we need to find an alternative way of reducing this problem to the one which we are able to treat. In order to overcome this problem, we will differentiate equation \eqref{4.fD} not in space, but in time. Indeed, denoting $w=\Dt u$ and $\theta=\Dt w$, we get
\begin{equation}\label{4.first}
\Dt w+f'_u(u,\Nx u)w+f'_{\Nx u}(u,\Nx u)\Nx w=\Dx w-w,\ \ w\big|_{x=0}=w\big|_{x=L}=0
\end{equation}
and
\begin{multline}\label{4.second}
\Dt\theta+f''_{u,u}(u,\Nx u)[w,w]+2f''_{u,\Nx u}(u,\Nx u)[w,\Nx w]+f''_{\Nx u,\Nx u}(u,\Nx u)[\Nx w,\Nx w]+\\+f'_u(u,\Nx u)\theta+f'_{\Nx u}(u,\Nx u)\Nx\theta=\Dx\theta-\theta,\ \ \theta\big|_{x=0}=\theta\big|_{x=L}=0.
\end{multline}
These equations have the form of \eqref{4.rred}, but with one essential difference. Namely, they still contain the terms $\Nx u$ and $\Nx w$ which prevent us from doing the change of the $\theta$-variable. In order to transform them into more convenient form, we need the following lemma.

\begin{lemma}\label{Lem4.dxu} Let the function $f$ be of class $C_0^\infty$. Then, there exists a sufficiently large constant $N=N(f)$ such that the following boundary value problem:
\begin{equation}\label{4.non1}
\Dx u-u-Nu-f(u,\Nx u)=h,\ \ u\big|_{x=0}=u\big|_{x=L}=0
\end{equation}
possesses a unique solution $u\in H^2(0,L)\cap H^1_0(0,L)$ for any $h\in L^2(0,L)$. Moreover, the solution operator $\Upsilon: h\to u$ is of class $C^\infty$ as an operator from $H^s(0,L)$ to $H^{s+2}(0,L)\cap H^1_0(0,L)$ for all $s\ge0$.
\end{lemma}
\begin{proof} Since the statement of the lemma is more or less standard, we restrict ourselves by verifying the uniqueness of a solution of this boundary value problem and the Lipschitz continuity of the solution operator only. Indeed, let $u_1$ and $u_2$ be two solutions of \eqref{4.non1} and let $\bar u:=u_1-u_2$ and $\bar h:=h_1-h_2$. Then, substructing the equations for $u_1$ and $u_2$, we get
\begin{equation}\label{4.dif}
\Dx\bar u-\bar u-N\bar u-l_1(x)\bar u-l_2(x) \Nx\bar u=\bar h, \ \ \bar u\big|_{x=0}=\bar u\big|_{x=L}=0,
\end{equation}
where
\begin{multline}
l_1:=\int_0^1f'_u(su_1+(1-s)u_2,s\Nx u_1+(1-s)\Nx u_2)\,ds,\\ l_2:=\int_0^1f'_{\Nx u} (su_1+(1-s)u_2,s\Nx u_1+(1-s)\Nx u_2)\,ds.
\end{multline}
Since $f$ is of class $C_0^\infty$, we have the estimates
$$
|l_1(x)|\le C_f,\ \ |l_2(x)|\le C'_f,
$$
where the constants $C_f$ and $C_f'$ depend only on $f$. Multiplying now equation \eqref{4.dif} by $-\bar u$ and integrating over $x\in(0,L)$, we end up with
$$
\frac12\|\bar u\|^2_{H^1}+(N-C_f-\frac{(C'_f)^2}2-\frac12)\|\bar u\|^2_{L^2}\le \frac12\|\bar h\|^2_{L^2}.
$$
Fixing now $N>C_f+\frac{(C'_f)^2}2+\frac12$, we get the estimate
$$
\|\bar u\|^2_{H^1}\le\|\bar h\|^2_{L^2}
$$
which immediately gives the uniqueness and well-posedness of the solution operator $U$. To verify its Lipschitz continuity as a map from $L^2$ to $H^2\cap H^1_0$, it is sufficient to express $\Dx \bar u$ from the equation and use the above obtained estimate for the $H^1$-norm for estimating the right-hand side. Thus, the desired Lipschitz continuity for $s=0$ is proved and for $s>0$ it can be verified using the bootstrapping arguments. So, the lemma is proved.
\end{proof}
Using the proved lemma, we may write the equation \eqref{4.fD} as follows
\begin{equation}
\Dx u-u-Nu-f(u,\Nx u)=w-Nu:=h,\ \ u\big|_{x=0}=u\big|_{x=L}=0
\end{equation}
and write
\begin{equation}\label{4.trans}
u=\Upsilon(w-Nu),\ \ \Nx u=\Nx \Upsilon(w-Nu):=\Bbb T_1(u,w).
\end{equation}
Crucial for us is the fact that, in contrast to the operator $\Nx: u\to\Nx u$, the operator $\Bbb T_1$ does not decrease the smoothness, but even increases it. Indeed, according to the lemma, this operator acts from $[H^s(0,L)]^{2m}$ to $[H^{s+1}(0,L)]^m\subset [W^{s,\infty}(0,L)]^m$. This allows us to replace $\Nx u$ by $\Bbb T_1(u,w)$ in equations \eqref{4.first} and \eqref{4.second} and obtain the terms which do not decrease smoothness.
\par
To express the derivative $\Nx w$ in analogous way, we differentiate \eqref{4.trans} in time using that $\Dt u=w$ and $\Dt w=\theta$ to obtain
\begin{multline}\label{4.trans1}
w=\Dt \Upsilon(w-Nu)=\Upsilon'(w-Nu)(\theta-Nw),\\ \Nx w=\Nx[\Upsilon'(w-Nu)(\theta-Nw)]:=\Bbb T_2(u,w,\theta),
\end{multline}
where the operator $\Bbb T_2$ acts from $[H^s(0,L)]^{3m}$ to $[H^{s+1}(0,L)]^m$ and, therefore, does not decrease smoothness as well (actually, this operator is even linear in $\theta$, but we will not need this fact in the sequel).
\par
Inserting the above operators into equations \eqref{4.fD}, \eqref{4.first} and \eqref{4.second}, we arrive at the desired system of nonlocal RDA equations:
\begin{equation}\label{4.izvr}
\begin{cases}
\Dt u+f(u,\Bbb T_1(u,w))=\Dx u-u,\ \ u\big|_{x=0}=u\big|_{x=L}=0,\\
\Dt w+f'_u(u,\Bbb T_1(u,w))w+\\\phantom{eggogeggogeggog}+f'_{\Nx u}(u,\Bbb T_1(u,w))\Bbb T_2(u,w,\theta)=\Dx w-w,\ \ w\big|_{x=0}=w\big|_{x=L}=0,\\
\Dt\theta+f''_{u,u}(u,\Bbb T_1(u,w))[w,w]+2f''_{u,\Nx u}(u,\Bbb T_1(u,w))[w,\Bbb T_2(u,w,\theta)]+\\\phantom{eggogeggogeggog}+f''_{\Nx u,\Nx u}(u,\Bbb T_1(u,w))[\Bbb T_2(u,w,\theta),\Bbb T_2(u,w,\theta)]+\\\phantom{eggog}+f'_u(u,\Bbb T_1(u,w))\theta+f'_{\Nx u}(u,\Bbb T_1(u,w))\Nx\theta=\Dx\theta-\theta,\ \ \theta\big|_{x=0}=\theta\big|_{x=L}=0.
\end{cases}
\end{equation}
We see that all of the nonlinearities in this system except of the term $f'_{\Nx u}(u,\Bbb T_1(u,w))\Nx\theta$ map $H^1(0,L)$ to $H^1(0,L)$ and, therefore, do not decrease smoothness. Moreover, the matrix $f'_{\Nx u}(u,\Bbb T_1(u,w))$ in the last term does not depend explicitly on $\theta$, so arguing as before, we may make it {\it small} by the transform $\theta=av$ which changes only the $\theta$ component and remains the $u$ and $w$ components unchanged. Thus, we have proved the following theorem.

\begin{theorem}\label{Th4.fin} Let the function $f$ be of class $C_0^\infty$. Then the system \eqref{4.fD} or RDA equations possesses an IM in the phase space $[H^1_0(0,L)]^m$ in the sense of Definition \ref{Def.IMext}.
\end{theorem}
\bibliography{References}

\end{document}